\documentclass[11pt]{amsart}
\usepackage[usenames,dvipsnames]{pstricks}
\usepackage[mathscr]{eucal}
\usepackage{amsfonts,amsmath,amssymb,amsthm,amscd,amsxtra}
\usepackage{enumerate,verbatim}
\usepackage[all,2cell,ps]{xy}
\usepackage[notcite]{}
\usepackage[pagebackref]{hyperref}
\usepackage{calc,color}

\theoremstyle{plain} 

\newtheorem{thm}{Theorem}[section]
\newtheorem{dfn}[thm]{Definition}

\newtheorem{eg}[thm]{Example}

\newtheorem{rmk}[thm]{Remark}
\newtheorem{prop}[thm]{Proposition}

\newtheorem{cor}[thm]{Corollary}
\newtheorem{lem}[thm]{Lemma}

\numberwithin{equation}{section}

\newcommand{\fm}{\mathfrak{m}}

\newcommand{\fp}{\mathfrak{p}}

\newcommand{\fa}{\mathfrak{a}}
\newcommand{\fb}{\mathfrak{b}}
\newcommand{\fc}{\mathfrak{c}}

\newcommand{\R}{\mathbf{R}}

\DeclareMathOperator{\ann}{ann} \DeclareMathOperator{\Ass}{Ass}
 
\DeclareMathOperator{\E}{E} \DeclareMathOperator{\hh}{H}
 \DeclareMathOperator{\Att}{Att}
\DeclareMathOperator{\cm}{NCM} \DeclareMathOperator{\X}{X}

\def\gd{\operatorname{\mathsf{G-dim}}}
\def\g{\operatorname{\mathsf{G}}}
\def\gkd{\operatorname{\mathsf{G}_{\it C}\mathsf{-dim}}}
\def\gkkd{\operatorname{\mathsf{G}_{\it K}\mathsf{-dim}}}
\def\gc{\operatorname{\mathsf{G}_{\it C}}}
\def\gcpd{\operatorname{\mathsf{G}_{\it C_p}\mathsf{-dim_{R_p}}}}
\def\pd{\operatorname{\mathsf{pd}}}
\def\rgr{\operatorname{\mathsf{r.grade}}}
\def\gr{\operatorname{\mathsf{grade}}}
\def\syz{\operatorname{\mathsf{syz}}}
\def\Tr{\mathsf{Tr}}
\def\trk{\mathsf{Tr}_{C}}
\def\trc{\mathcal{T}^C}
\def\trr{\mathcal{T}}
\def\trcp{\mathsf{Tr}_{C_\fp}}
\def\Min{\mathsf{Min}}
\def\cc{\mathsf{c}}
\def\ng{\mathsf{NG}_{C}}

\DeclareMathOperator{\coker}{Coker}

\def\depth{\operatorname{\mathsf{depth}}}
 
\def\Ext{\operatorname{\mathsf{Ext}}}

\def\Hom{\operatorname{\mathsf{Hom}}}

\DeclareMathOperator{\id}{id} 

\DeclareMathOperator{\Supp}{Supp} \DeclareMathOperator{\Spec}{Spec}

\def\Tor{\operatorname{\mathsf{Tor}}}

\def\urltilda{\kern -.15em\lower .7ex\hbox{\~{}}\kern .04em}
\def\urldot{\kern -.10em.\kern -.10em}\def\urlhttp{http\kern -.10em\lower -.1ex
\hbox{:}\kern -.12em\lower 0ex\hbox{/}\kern -.18em\lower
0ex\hbox{/}}

\begin{document}
\baselineskip=15pt

\title[Linkage of modules and the Serre conditions]
 {Linkage of modules and the Serre conditions }

\bibliographystyle{amsplain}

\author[M. T. Dibaei]{Mohammad T. Dibaei$^{1}$}
\author[A. Sadeghi]{Arash Sadeghi$^2$}\

\address{$^{1}$ Faculty of Mathematical Sciences and Computer,
Kharazmi University, Tehran, Iran.}

\address{$^{1, 2}$ School of Mathematics, Institute for Research in Fundamental Sciences (IPM), P.O. Box: 19395-5746, Tehran, Iran }
\email{dibaeimt@ipm.ir} \email{sadeghiarash61@gmail.com}

\keywords{Auslander class, linkage of modules, semidualizing
modules, $\gc$--dimension modules.\\
1. M.T. Dibaei was supported in part by a grant from IPM (No.
93130110)} 
 \maketitle
\begin{center} {\it To the memory of Jan Strooker}\end{center}
\begin{abstract}
Let $R$ be semiperfect commutative Noetherian ring and $C$ be a
semidualizing $R$--module. The connection of the Serre condition
$(S_n)$ on a horizontally linked $R$-module of finite
$\gc$-dimension with the vanishing of certain cohomology modules of
its linked module is discussed. As a consequence, it is shown that
under some conditions Cohen-Macaulayness is preserved under
horizontal linkage.
\end{abstract}

\section{introduction}
The theory of linkage of algebraic varieties was introduced by
Peskine and Szpiro \cite{PS}. Recall that two ideals $\fa$ and $\fb$
in a Cohen-Macaulay local ring $R$ are said to be linked if there is
a regular sequence $\alpha$ in their intersection such that
$\fa=(\alpha):\fb$ and $\fb=(\alpha):\fa$. The first main theorem in
the theory of linkage was due to C. Peskine and L. Szpiro. They
proved that over a Gorenstein local ring $R$ with linked ideals
$\fa$ and $\fb$, $R/\fa$ is Cohen-Macaulay if and only if $R/\fb$
is. They also gave a counter-example to show that the above
statement is no longer true if the base ring is Cohen-Macaulay but
non-Gorenstein. In \cite[Theorem 4.1]{Sc}, Schenzel proved that,
over a Gorenstein local ring $R$ with maximal ideal $\fm$, the Serre
condition $(S_r)$ for $R/{\fa}$ is equivalent to the vanishing of
the local cohomology groups $\hh^i_{\fm}(R/{\fb})= 0$ for all $i$,
$\dim(R/{\fb})-r<i<\dim(R/{\fb})$, provided $\fa$ an $\fb$ are
linked by a Gorenstein ideal $\fc$ (i.e.
 $\fc\subseteq\fa\cap\fb$, $\fa=\fc:_R\fb$ and $\fb=\fc:_R\fa$, equivalently, the ideals
$\fa/\fc$ and $\fb/\fc$ of the ring $R/\fc$ are linked by the zero
ideal).

In \cite{MS}, Martsinkovsky and Strooker generalized the notion of
linkage for modules over non-commutative semiperfect Noetherian
rings over which any finitely generated module admits a projective
cover. They introduced the operator $\lambda=\Omega\Tr$ and showed
that ideals $\fa$ and $\fb$ are linked by zero ideal if and only if
$R/\fa\cong\lambda (R/\fb)$ and $R/\fb\cong\lambda (R/\fa)$
\cite[Proposition 1]{MS}.

The present authors, in \cite[Theorem 4.2]{DS}, extended Schenzel's
result for any horizontally linked module of finite $\g$-dimension
over a more general ground ring, i.e. over a Cohen-Macaulay local
ring. More precisely, for a horizontally linked module $M$ of finite
Gorenstein dimension over a Cohen-Macaulay local ring $R$, it is
shown that $M$ satisfies the Serre condition $(S_n)$ for some $n>0$
if and only if $\hh^i_\fm(\lambda M)=0$ for all $i$, $\dim
R-n<i<\dim R$. In this paper, we continue our study about the effect
of the Serre condition on a horizontally linked module and extend
Schenzel's result in different directions.

Now we describe the organization of the paper. In Section 2, after
preliminary notions and definitions, we generalize a result of
Auslander and Bridger \cite[Theorem 4.25]{AB} for modules in the
Auslander class $\mathcal{A}_C$, with respect to a semidualizing
module $C$ (see Theorem \ref{th5}). As a consequence, we give a
generalization of Schenzel's result for modules in the Auslander
class with respect to a semidualizing module. For a semidualizing
$R$--module $C$ and a stable $R$--module $M$, it is shown that if
$M\in\mathcal{A}_C$ and $\id_{R_\fp}(C_\fp)<\infty$ for all $\fp\in
\Spec R$ with $\depth R_\fp\leq n-1$, then $M$ satisfies
$\widetilde{S}_n$ (i.e.
 $\depth_{R_\fp} (M_\fp)\geq \min\{n, \depth R_\fp\}$ for all
 $\fp\in\Supp_R(M)$) if and only if $M$ is horizontally linked and
$\Ext^i_R(\lambda M,C)=0$ for all $i$, $0<i<n$ (see Corollary
\ref{cor7}).

In Section 3, we study the theory of linkage for maximal
Cohen-Macaulay modules (mCM). We present a characterization of a
mCM--module whose linked module is also mCM (Theorem \ref{the1}). As
a consequence, we obtain the following result for the linkage of
ideals. Assume that $R$ is a Cohen-Macaulay local ring of dimension
$d$ with canonical module $\omega_R$ which is generically Gorenstein
so that $\omega_R$ is identified with an ideal of $R$. It is shown
that if an ideal $I$ is linked to an ideal $J$ by zero ideal such
that $I\omega_R=I\cap\omega_R$ and $R/I$ is Cohen-Macaulay, then
$R/J$ is Cohen-Macaulay if and only if $R/I+\omega_R$ is
Cohen-Macaulay of dimension $d-1$ (see Corollary \ref{theorem3}).

 In Section 4, for a semidualizing module
$C$, we study the theory of linkage for modules of finite
$\gc$-dimension which is inspired by authors pervious work \cite{DS}
on the theory of linkage for modules of finite $\g$-dimension. Let
$R$ be a Cohen-Macaulay local ring, $C$ a semidualizing $R$-module
and $c_1$, $c_2$ two $\gc$-Gorenstein ideals. Assume that $M_1$,
$M$, and $M_2$ are $R$-modules such that
$M_1\underset{\fc_1}{\thicksim}M\underset{\fc_2}{\thicksim}M_2$. In
Theorem \ref{prop}, it is shown that if $\gkd_R(M)<\infty$, then
$M_1$ is Cohen-Macaulay if and only if $M_2$ is. For a horizontally
linked module $M$ of finite $\gc$-dimension and a positive integer
$n$ such that $\lambda M\in \mathcal{A}_C$ (e.g. $\pd_R(\lambda
M)<\infty$), it is shown that $\lambda M$ satisfies
$\widetilde{S}_n$ if and only if $\Ext^i_R(M,C)=0$ for all $i$,
$0<i<n$ (see Theorem \ref{th1}), which generalizes \cite[Proposition
2.6]{DS} and also can be viewed as a generalization of \cite[Theorem
4.1]{Sc}. As a consequence, for a Cohen-Macaulay local ring $R$ and
a horizontally linked module $M$ of finite $\gc$-dimension such that
$\lambda M\in \mathcal{A}_C$, $M$ is mCM if and only if $\lambda M$
is (see Corollary \ref{cor5}). In Corollary \ref{cor2}, for a
horizontally linked $R$--module $M$ over a Gorenstein local ring
$R$, we determine the attached primes of the local cohomology module
$\hh^{\tiny{\cc(M)}}_\fm(M)$, in terms of the depth of $\lambda M$,
where $\cc(M)$ denotes the greatest integer $n(<\dim_R(M))$ such
that $\hh^{n}_\fm(M)\neq0$. As a consequence, it is shown that
$\hh^{\tiny{\cc(M)}}_\fm(M)$ is finitely generated if and only if
$\depth_R(\lambda M)+\cc(M)=\depth R$ and $\depth_{R_\fp}((\lambda
M)_\fp)+\cc(M)>\depth R$ for all $\fp\in\cm(M)\setminus\{\fm\}$,
where $\cm(M)$ denotes the non--Cohen-Macaulay locus of $M$ (see
Theorem \ref{cor3}).

In Section 5, we study the theory of linkage for reduced
$\gc$-perfect module (Definition \ref{rgr}). These modules can be
viewed as a generalization of the Eilenberg-Maclane module \cite{H}.
For a reduced $\gc$-perfect module $M$ of $\gc$-dimension $n$ over a
Cohen-Macaulay local ring $R$, it is shown in Theorem \ref{th4} that
if $\lambda M\in\mathcal{A}_C$, then $\depth_R(M)+\depth_R(\lambda
M)=\depth R+\depth_R(\Ext^n_R(M,C))$ which is a generalization of
\cite[Theorem 3.3]{DS}. We end the paper by determining conditions
under which an
Eilenberg-Maclane horizontally linked $R$--module is generalized Cohen-Macaulay (see Corollary \ref{cor1}).\\

Throughout the paper, $R$ is a commutative Noetherian ring and all
$R$--modules $M$, $N$, $\cdots$ are finite (i.e. finitely
generated). Whenever, $R$ is assumed local, its unique maximal ideal
is denoted by $\fm$.

For a finite presentation $P_1\overset{f}{\rightarrow}P_0\rightarrow
M\rightarrow 0$ of an $R$--module $M$, its transpose, $\Tr M$, is
defined as $\coker f^*$, where $(-)^* := \Hom_R(-,R)$, which
satisfies in the exact sequence
\begin{equation}\label{1.1}
0\rightarrow M^*\rightarrow P_0^*\rightarrow P_1^*\rightarrow \Tr
M\rightarrow 0.
\end{equation}
Moreover, $\Tr M$ is unique up to projective equivalence. Thus all
minimal projective presentations of $M$ represent isomorphic
transposes of $M$. For two $R$--modules $M$ and $N$, there exists
the following exact sequence
\begin{equation}\label{1.2}
0\longrightarrow \Ext^1_R(\Tr M,N)\longrightarrow
M\underset{R}{\otimes}N\overset{e_M^N}{\longrightarrow}\Hom_R(M^*,N)
\longrightarrow\Ext^2_R(\Tr M,N)\longrightarrow0,
\end{equation}
where $e_M^N:M\underset{R}{\otimes}N\rightarrow \Hom_R(M^*,N)$ is
the evaluation map \cite[Proposition 2.6]{AB}.

The syzygy of a module $M$, denoted by $\Omega M$, is the kernel of
an epimorphism $P\overset{\alpha}{\rightarrow}M$, where $P$ is a
projective $R$--module, so that it is unique up to projective
equivalence. Thus $\Omega M$ is uniquely determined, up to
isomorphism, by a projective cover of $M$.

Martsinkovsky and Strooker  \cite{MS} generalized
the notion of linkage for modules over non-commutative semiperfect
Noetherian rings (i.e. finitely generated modules over such rings have projective covers).
\begin{dfn}\label{d1}\cite[Definition 3]{MS}
Let $R$ be a semiperfect ring. Two $R$--modules $M$ and $N$ are said
to be\emph{ horizontally linked} if $M\cong \lambda N$ and
$N\cong\lambda M$. Also, $M$ is called horizontally linked (to
$\lambda M$) if $M\cong\lambda^2M$.
\end{dfn}
Note that a commutative ring $R$ is semiperfect if and only if it is
a finite direct product of commutative local rings \cite[Theorem
23.11]{L}. A \emph{stable} module is a module with no non-zero
projective direct summands. Let $R$ be a semiperfect ring, $M$ a
stable $R$--module and $P_1\rightarrow P_0\rightarrow M\rightarrow
0$ a minimal projective presentation of $M$. Then $P_0^*\rightarrow
P_1^*\rightarrow \Tr M\rightarrow 0$ is a minimal projective
presentation of $\Tr M$ \cite[Theorem 32.13]{AF}. The following
induced exact sequences
\begin{equation}\tag{\ref{d1}.1}
0\longrightarrow M^*\longrightarrow P_0^*\longrightarrow \lambda
M\longrightarrow0,
\end{equation}
\begin{equation}\tag{\ref{d1}.2}
0\longrightarrow \lambda M\longrightarrow P_1^*\longrightarrow \Tr
M\longrightarrow0,
\end{equation}
which will be quoted in this paper.

An $R$--module $M$ is called a \emph{syzygy module} if it is
embedded in a projective $R$--module. Let $i$ be a positive integer,
an $R$--module $M$ is said to be an $i$th syzygy if there exists an
exact sequence
$$0\rightarrow M\rightarrow P_{i-1}\rightarrow\cdots\rightarrow P_0$$ with the $P_0,\cdots,P_{i-1}$ are
projective. By convention, every module is a $0$th syzygy.

A characterization of a module to be horizontally linked, involving
a syzygy property, is given bellow.

\begin{thm}\cite[Theorem 2 and Proposition 3]{MS}\label{MS}
Let $R$ be a semiperfect ring. An $R$--module $M$ is horizontally
linked if and only if it is stable and $\Ext^1_R(\Tr M,R)=0$,
equivalently $M$ is stable and a syzygy module.
\end{thm}
\begin{dfn}\label{d3}
An $R$--module $C$ is called a \emph{semidualizing} module, if the homothety morphism $R\rightarrow\Hom_R(C,C)$ is
an isomorphism and $\Ext^i_R(C,C)=0$ for all $i>0$.
\end{dfn}
Semidualizing modules are initially studied in \cite{F} and
\cite{G}. It is clear that $R$ itself is a semidualizing
$R$--module. Over a Cohen-Macaulay local ring $R$, a canonical
module $\omega_R$ of $R$ is a semidualizing module with finite
injective dimension.

Let $C$ be a semidualizing $R$--module, $M$ an $R$--module. Let
$P_1\overset{f}{\rightarrow}P_0\rightarrow M\rightarrow 0$ be a
projective presentation of $M$. The transpose of $M$ with respect to
$C$, $\trk M$, is defined to be $\coker f^{\triangledown}$, where
$(-)^{\triangledown} := \Hom_R(-,C)$, which satisfies the exact
sequence
\begin{equation}\tag{\ref{d3}.1}
0\rightarrow M^{\triangledown}\rightarrow
P_0^{\triangledown}\rightarrow P_1^{\triangledown}\rightarrow \trk
M\rightarrow 0.
\end{equation}
By \cite[Proposition 3.1]{F}, there exists the following exact
sequence
\begin{equation}\tag{\ref{d3}.2}
0\rightarrow\Ext^1_R(\trk M,C)\rightarrow M\rightarrow
M^{\triangledown\triangledown}\rightarrow\Ext^2_R(\trk
M,C)\rightarrow0.
\end{equation}
Therefore, one has the following exact sequence (see for example
\cite[Theorem 2.4]{Z})
\begin{equation}\tag{\ref{d3}.3}
0\rightarrow\Ext^1_R(M,C)\rightarrow\trk M\rightarrow (\trk
M)^{\triangledown\triangledown}\rightarrow\Ext^2_R(M,C)\rightarrow0.
\end{equation}

The Gorenstein dimension has been extended to $\gc$--dimension by
Foxby in \cite{F} and by Golod in \cite{G}.
\begin{dfn}
An $R$--module $M$ is said to have $\gc$--dimension zero if $M$ is
$C$-reflexive, i.e. the canonical map $M\rightarrow
M^{\triangledown\triangledown}$ is bijective and
$\Ext^i_R(M,C)=0=\Ext^i_R(M^{\triangledown},C)$ for all $i>0$.
\end{dfn}
A $\gc$-resolution of an $R$--module $M$ is a right acyclic complex
of $\gc$-dimension zero modules whose $0$th homology is $M$. The
module $M$ is said to have finite $\gc$-dimension, denoted by
$\gkd_R(M)$, if it has a $\gc$-resolution of finite length. From the
exact sequences (\ref{d3}.1) and (\ref{d3}.2), it is clear that
$\gkd_R(M)=0$ if and only if $\Ext^i_R(\trk M,C)=0=\Ext^i_R(M,C)$
for all $i>0$. Hence one has the following conclusion.

\begin{rmk}\label{rem1}
Let $C$ be a semidualizing $R$--module. For an $R$--module $M$,
$\gkd_R(M)=0$ if and only if $\gkd_R(\trk M)=0$.
\end{rmk}

\begin{proof}
It is straightforward by using exact sequences (\ref{d3}.1),
(\ref{d3}.2) and (\ref{d3}.3).
\end{proof}
Note that, over a local ring $R$, a semidualizing $R$--module $C$ is
a canonical module if and only if $\gkd_R(M)<\infty$ for all
finitely generated $R$--modules $M$ (see \cite[Proposition
1.3]{Ge}).

Recall that the $\gc$-dimension of a module $M$, if finite, can be
expressed as follows.
\begin{thm}(\cite[4,8]{G})\label{G3}
For a semidualizing $R$--module $C$ and an $R$--module $M$ of finite
$\gc$--dimension, the following statements hold true.
\begin{itemize}
       \item[(i)]{$\gkd_R(M)=\sup\{i\mid\Ext^i_R(M,C)\neq0, i\geq0\}$,}
        \item[(ii)]{If $R$ is local, then $\gkd_R(M)=\depth R-\depth_R(M)$.}
\end{itemize}
\end{thm}

\section{Linkage and Serre conditions}
For ideals $\fa$ and $\fb$ in a Gorenstein local ring $R$ which are
linked by a Gorenstein ideal $\fc$, Schenzel in \cite[Theorem
4.1]{Sc} proved that $R/\fa$ satisfies $(S_r)$ if and only if
$\hh^i_\fm(R/\fb)=0$ for all $i$, $\dim R/\fb-r<i<\dim R/\fb$. In
order to develop this theory for modules, we will first generalize
the result \cite[Theorem 4.25]{AB} of Auslander and Bridger for
$\trk M$, the transpose of a module $M$ with respect to a
semidualizing module $C$ as in Proposition \ref{t1}. We then express
Schenzel's result for modules, over Cohen-Macaulay local rings with
canonical module, in two forms as in Proposition \ref{p3} and
Corollary \ref{c2}.

In another direction, we extend \cite[Theorem 3.3(i)]{Ka} of Kawasaki and \cite[Theorem 1.11]{KY}, of Khatami and Yassemi for modules in the Auslander class
with respect to a semidualizing module.

Throughout, we fix an $R$--module $C$ and denote $(-)^\triangledown$
as the dual functor $(-)^\triangledown:=\Hom_R(-,C)$. Two
$R$--modules $M$ and $N$ are said to be stably equivalent with
respect $C$, denoted $M \underset{C}{\approx} N$, if $C^n\oplus
M\cong C^m\oplus N$ for some non-negative integers $m$ and $n$. We
write $M \approx N$ when $M$ and $N$ are stably equivalent with
respect $R$. For $k>0$, the composition
$\mathcal{T}_k:=\Tr\Omega^{k-1}$ were introduced by Auslander and
Bridger in \cite{AB}.

\begin{rmk}\label{remark3}
\begin{enumerate}[(i)]
\item{Let $R$ be a semiperfect ring, $M$ an $R$--module. Assume that
$P_1\rightarrow P_0\rightarrow M\rightarrow0$ is the minimal
projective presentation of $M$. There exists a commutative diagram
$$\begin{CD}
&&&&&&&&\\
  \ \ &&&&  P^*_0\underset{R}{\otimes}C @>>>P^*_1\underset{R}{\otimes}C @>>> \Tr M\underset{R}{\otimes}C @>>>0&  \\
  &&&& @VV{\cong}V @VV{\cong}V \\
  \ \  &&&& P^{\triangledown}_0 @>>> P^{\triangledown}_1 @>>>\trk M @>>>0&\\
\end{CD}$$\\
with exact rows. Therefore, $\Tr M\underset{R}{\otimes}C\cong\trk
M$.}
\item{ For an $R$--module $M$ and positive integer $n$, there
is an exact sequence
$$0\rightarrow\Ext^n_R(M,R)\rightarrow\mathcal{T}_nM\rightarrow\lambda^2\mathcal{T}_nM\rightarrow0 \text{(see \cite[Remark 1.9]{DS})}.$$
Note that $\lambda^2\mathcal{T}_nM\approx\Omega\mathcal{T}_{n+1}M$.
Hence if $\Ext^n_R(M,R)=0$, then
$\mathcal{T}_nM\approx\Omega\mathcal{T}_{n+1}M$.}
\end{enumerate}
\end{rmk}

The proof of the following lemma is based on the proof of
\cite[Lemma 3.9]{AB}.
\begin{lem}\label{l2}
Let $C$ be a semidualizing $R$--module. If $$0\rightarrow
M_1\rightarrow M_2\rightarrow M_3\rightarrow0$$ is an exact sequence
of $R$--modules then there exists a natural long exact sequence
$$0\rightarrow M^{\triangledown}_3\rightarrow
M^{\triangledown}_2\rightarrow M^{\triangledown}_1\rightarrow \trk
M_3 \rightarrow\trk M_2\rightarrow\trk M_1\rightarrow0.$$
\end{lem}

Let $C$ be a semidualizing $R$--module, $M$ an $R$--module. For a
generating set $\{f_1,f_2,\ldots,f_n\}$ of
$M^\triangledown=\Hom_R(M,C)$, denote $f:M\rightarrow C^n$ as the
map $(f_1,\ldots,f_n)$. It follows from (\ref{d3}.2) that $f$ is
injective if and only if $\Ext^1_R(\trk M,C)=0$. Note that when $f$
is injection, then there is an exact sequence
\begin{equation}\tag{\ref{l2}.1}
0\rightarrow M\overset{f}{\rightarrow} C^n\rightarrow N\rightarrow0,
\end{equation}
where $N=\coker(f)$. It is easy to see that, in this situations, the
exact sequence (\ref{l2}.1) is dual exact with respect to
$(-)^\triangledown$ and so $\Ext^1_R(N,C)=0$. Such an exact sequence
is called a \emph{universal pushforward of $M$ with respect to $C$}.

\begin{dfn}\cite{M1}\label{S}
An $R$--module $M$ is said to satisfy the property $\widetilde{S}_k$
if $\depth_{R_\fp} (M_\fp) \geq  \min\{ k, \depth R_\fp\}$  for all
$\fp\in\Spec R$.
\end{dfn}
Note that, for a horizontally linked module $M$ over a
Cohen-Macaulay local ring $R$, the properties $\widetilde{S}_k$ and
$(S_k)$ are identical.

For a positive integer $n$, a module $M$ is called an $n$th
$C$-syzygy module if there is an exact sequence $0\rightarrow
M\rightarrow C_1\rightarrow C_2\rightarrow\ldots\rightarrow C_n$,
where $C_i\cong\oplus^{m_i}C$ for some $m_i$.

Let $X$ be a subset of $\Spec R$. An $R$--module $M$ is said to be
of finite $\gc$--dimension on $X$, if $\gcpd(M_{\fp})<\infty$ for
all $\fp\in X$. We denote $X^n(R):=\{\fp\in\Spec(R)\mid \depth
R_\fp\leq n\}$.

Recall that an $R$--module $M$ is $n$-torsion free if $\Ext^i_R(\Tr
M,R)=0$ for all $1\leq i\leq n$. In \cite[Theorem 4.25]{AB},
Auslander and Bridger proved that an $R$--module $M$ of finite
Gorenstein dimension is $n$-torsion free if and only if $M$
satisfies $\widetilde{S}_n$ (See also \cite[Theorem 42]{M1}). In the
following we generalize this result in different directions.

\begin{prop}\label{t1}
Let $C$ be a semidualizing $R$--module and $M$ an $R$--module. For a
positive integer $n$, consider the following statements.
\begin{itemize}
      \item[(i)]$\Ext^i_R(\trk M,C)=0$ for all $i$, $1\leq i\leq n$.
      \item[(ii)]$M$ is an $n$th $C$-syszygy module.
      \item[(iii)]$M$ satisfies $\widetilde{S}_n$.
\end{itemize}
Then the following implications hold true.
\begin{itemize}
       \item[(a)] (i)$\Rightarrow$(ii)$\Rightarrow$(iii).
       \item[(b)] If $M$ has finite $\gc$--dimension on $X^{n-1}(R)$, then (iii)$\Rightarrow$(i).
\end{itemize}
\end{prop}
\begin{proof}
(a).(i)$\Rightarrow$(ii) Applying $(-)^{\triangledown}:=\Hom_R(-,C)$
to a projective resolution $\cdots\rightarrow
P_{n-1}\rightarrow\cdots\rightarrow P_0\rightarrow
M^{\triangledown}\rightarrow0$ of $M^{\triangledown}$, implies the
complex $0\rightarrow M^{\triangledown\triangledown}\rightarrow
(P_0)^{\triangledown}\rightarrow\cdots\rightarrow
(P_{n-2})^{\triangledown}\rightarrow(P_{n-1})^{\triangledown}.$ Note
that the exact sequence (\ref{d3}.2) implies that $M$ is embedded in
$M^{\triangledown\triangledown}$ if $n=1$ and $M\cong
M^{\triangledown\triangledown}$ if $n>1$. Therefore $M$ is always
$1$st $C$--syzygy and, for $n=2$, $M$ is $2$nd $C$-syzygy. Assuming
$n>2$ implies $\Ext^i_R(\trk
M,C)\cong\Ext^{i-2}_R(M^\triangledown,C)$ for all $i$, $2<i\leq n$,
by the exact sequence (\ref{d3}.1). Therefore the complex
$0\rightarrow M^{\triangledown\triangledown}\rightarrow
(P_0)^{\triangledown}\rightarrow\cdots\rightarrow
(P_{n-2})^{\triangledown}\rightarrow(P_{n-1})^{\triangledown}$ is
exact, i.e. $M$ is an $n$th $C$-syzygy.

(ii)$\Rightarrow$(iii). By assumption there is an exact sequence
$0\longrightarrow M\longrightarrow
C_1\longrightarrow\cdots\longrightarrow C_n$, where $C_i=C^{l_i}$
for some positive integer $l_i$, $i=1, \cdots, n$. For any
$\fp\in\Spec R$, one has  $\depth_{R_\fp}
(M_\fp)\geq\min\{\depth_{R_\fp}(C_n)_\fp, n\}$. By \cite[page
63(1)]{G}, $\depth_{R_\fp}(C_n)_\fp=\depth_{R_\fp}(C_{\fp})=\depth
R_{\fp}$ for all $\fp\in\Spec R$ and so we get the result.

(b). We argue by induction on $n$. If $n=1$ then, by Theorem
\ref{G3}, $M$ is of $\gc$--dimension zero on $X^{0}(R)$ and so, by
Remark \ref{rem1}, $\trk M$ is of $\gc$--dimension zero on
$X^{0}(R)$. Hence $\Ext_R^1(\trk M, C)_\fp=0$ for all $\fp\in
X^{0}(R)$ by Theorem \ref{G3}(i). It is enough to show that
$\Ass_R(\Ext^1_R(\trk M,C))=\emptyset$. Assume contrarily that
$\fp\in\Ass_R(\Ext^1_R(\trk M,C))$. By the exact sequence
(\ref{d3}.2), $\depth_{R_\fp}(M_\fp)=0$. As $M$ satisfies
$\widetilde{S}_1$, $\fp\in X^{0}(R)$, which is a contradiction.

Now, let $n>1$. By Theorem \ref{G3},
$\g_{C_\fp}$-$\dim_{R_\fp}(M_\fp)= \depth
R_\fp-\depth_{R_\fp}(M_\fp)$ for all $\fp\in\X^{n-1}(R)$. As $M$
satisfies $\widetilde{S}_{n-1}$, $M$ is of $\gc$--dimension $0$ on
$\X^{n-1}(R)$ and so $\Ext^1_R(\trk M, C)=0$. As noted, after Lemma
\ref{l2}, there is a universal pushforward
\begin{equation}\tag{\ref{t1}.1}
0\longrightarrow M\longrightarrow C^m\longrightarrow
N\longrightarrow0.
\end{equation}
of $M$ with respect to $C$. As $\Ext^1_R(N,C)=0$, the exact sequence
(\ref{t1}.1) implies the exact sequence
\begin{equation}\tag{\ref{t1}.2}
0\longrightarrow\trk N\longrightarrow C^m\longrightarrow\trk M\longrightarrow0,
\end{equation}
by Lemma \ref{l2}. Note that $N$ has finite
$\gc$--dimension on $X^{n-1}(R)$. As $M$ satisfies $\widetilde{S}_n$
and $\Ext^1_R(N,C)=0$, it is easy to see that $N$ satisfies
$\widetilde{S}_{n-1}$. Induction hypothesis on $N$, gives that
$\Ext^i_R(\trk N,C)=0$ for all $i$, $1\leq i\leq n-1$. Finally, the
exact sequence (\ref{t1}.2) implies that $\Ext^i_R(\trk M,C)=0$ for
all $i$, $1\leq i\leq n$.
\end{proof}

As an application of Proposition \ref{t1}, we can find the effect of
$M\underset{R}{\otimes}\omega_R$ having the Serre condition $(S_n)$
on the vanishing of the local cohomology groups of $\lambda M$ when
$M$ is a horizontally linked module over a Cohen-Macaulay local ring
$R$ which admits a canonical module $\omega_R$.  This study brings,
in particular, some generalizations to the result of Schenzel
\cite[Theorem 4.1]{Sc}.
First we recall the Local Duality Theorem \cite[Corollary
3.5.9]{BH}.
\begin{thm}\label{th8}
Let $(R,\fm,k)$ be a Cohen-Macaulay local ring of dimension $d$ with
a canonical module $\omega_R$. Then for all finite $R$--modules $M$
and all integers $i$ there exist natural isomorphisms
$$\hh^i_{\fm}(M)\cong\Hom_R(\Ext^{d-i}_R(M,\omega_R),\E_R(k)),$$
where $\E_R(k)$ is the injective envelope of $k$.
\end{thm}

\begin{prop}\label{p3}
Let $R$ be a Cohen-Macaulay local ring of dimension $d$ with
canonical module $\omega_R$, $M$ a horizontally linked $R$--module.
Suppose that $M\underset{R}{\otimes}\omega_R$ satisfies $(S_1)$.
Then, for a positive integer $n$, the following statements are
equivalent.
\begin{itemize}
             \item[(i)]$\lambda M$ satisfies $(S_n)$.
              \item[(ii)] $\hh^i_{\fm}(M\underset{R}{\otimes}\omega_R)=0$ for all $i$, $d-n<i<d$.
\end{itemize}
\end{prop}
\begin{proof}
By writing (\ref{1.2}), in terms of $\omega_R$, we may consider
$\Ext^1_R(\Tr M,\omega_R)$ as a submodule of
$M\underset{R}{\otimes}\omega_R$. From the fact that
$M\underset{R}{\otimes}\omega_R$ satisfies $(S_1)$, it follows that
$\Ass_R(\Ext^1_R(\Tr M,\omega_R))\subseteq\Min R$ which gives
$\Ext^1_R(\Tr M,\omega_R)=0$.

Note that, as $\Ass_R(\lambda M)\subseteq\Ass R$,  $\lambda M$
satisfies $(S_n)$ if and only if $\lambda M$ satisfies the condition
$\widetilde{S}_n$ which in turn is equivalent to say $\Tr M$
satisfies $\widetilde{S}_{n-1}$ because $\lambda M$ is the first
syzygy of $\Tr M$ and $\Ext^1_R(\Tr M,\omega_R)=0$. Thus, over the
Cohen-Macaulay ring $R$ and by Proposition \ref{t1}, the statement
(i) is equivalent to
\begin{equation}\tag{\ref{p3}.1}
\Ext^i_R(\Tr_{\omega_R}(\Tr M),\omega_R)=0 \text{ for all } i, 1\leq
i\leq n-1.
\end{equation}
On the other hand, Remark \ref{remark3}(i) shows that
$\Tr_{\omega_R}(\Tr M)\cong\Tr\Tr M\underset{R}{\otimes}\omega_R$.
By Theorem \ref{MS}, $M$ is stable, and so $M\cong\Tr\Tr M$ .
Therefore, $\Tr_{\omega_R} (\Tr M)\cong
M\underset{R}{\otimes}\omega_R$. Hence $\lambda M$ satisfies $(S_n)$
if and only if $\Ext^i_R(M\underset{R}{\otimes}\omega_R,\omega_R)=0$
for all $i$, $1\leq i\leq n-1$, which is also equivalent to say that
$\hh^i_{\fm}(M\underset{R}{\otimes}\omega_R)=0$ for all $i$,
$d-n<i<d$, by the Local Duality Theorem \ref{th8}.
\end{proof}
To achieve another generalization related to \cite{Sc}, we need the
following which is analogous to Proposition \ref{t1}.
\begin{prop}\label{t13}
Let $C$ be a semidualizing $R$--module and $M$ an $R$--module. For a
positive integer $n$, consider the following statements.
\begin{itemize}
      \item[(i)]$\Ext^i_R(\Tr M,C)=0$ for all $i$, $1\leq i\leq n$.
      \item[(ii)]$M\underset{R}{\otimes}C$ is an $n$th $C$-syzygy.
      \item[(iii)]$M\underset{R}{\otimes}C$ satisfies $\widetilde{S}_n.$
\end{itemize}
Then the following statements hold true.
\begin{itemize}
         \item[(a)](i)$\Rightarrow$(ii)$\Rightarrow$(iii).
          \item[(b)] If $C$ has finite injective dimension on $X^{n-1}(R)$, then (iii) implies (i).
\end{itemize}
\end{prop}
\begin{proof}
Set $N=M\underset{R}{\otimes}C$. Note that as $C$ is semidualizing,
there are natural isomorphisms $N^\triangledown\cong M^*$ and
$N^{\triangledown\triangledown}\cong M^{*\triangledown}$. Therefore,
from the exact sequences (\ref{1.2}) and (\ref{d3}.2), one obtains
the commutative diagram

$$\begin{CD}
0\rightarrow\Ext^1_R(\Tr M,C)@>>>N@>>>\Hom_R(M^*,C)@>>>\Ext^2_R(\Tr M,C)\rightarrow0  \\
@VVV@VV{\parallel}V@VV{\|\wr}V@VVV \\
0\rightarrow\Ext^1_R(\trk(N),C)@>>>N@>>>(N)^{\triangledown\triangledown}
@>>>\Ext^2_R(\trk(N),C)\rightarrow0\\
\end{CD}
$$
\\
with exact rows, and so $\Ext^i_R(\trk(N),C)\cong\Ext^i_R(\Tr M,C)$
for $i=1, 2$.

It follows from the exact sequences (\ref{d1}.1), (\ref{d1}.2),  and
(\ref{d3}.1) that
 \[\begin{array}{rl}
\Ext^i_R(\Tr M,C)&\cong\Ext^{i-2}_R(M^*,C)\\
&\cong\Ext^{i-2}_R(N^\triangledown, C)\\
&\cong\Ext^i_R(\trk(N),C)
\end{array}\]
for all $i>2$. Now, by replacing $M$ by $N$ in the Proposition
\ref{t1}, the assertion follows.
\end{proof}
Now we are ready to give another generalization of \cite[Theorem
4.1]{Sc}.
\begin{cor}\label{c2}
Let $R$ be Cohen-Macaulay local ring of dimension $d$ with canonical
module $\omega_R$, $M$ a horizontally linked $R$--module such that
$M\underset{R}{\otimes}\omega_R$ satisfies $(S_1)$. For a positive
integer $n$, the following statements are equivalent.
\begin{itemize}
       \item[(i)]$M\underset{R}{\otimes}\omega_R$ satisfies $(S_n)$.
        \item[(ii)]$\hh^i_\fm(\lambda M)=0$ for all $i$, $d-n<i<d$.
\end{itemize}
\end{cor}
\begin{proof}
 As in the proof of Proposition \ref{p3}, the fact that $M\underset{R}{\otimes}\omega_R$ satisfies $(S_1)$ implies that $\Ext^1_R(\Tr
M,\omega_R)=0$. Now the assertion is clear by Proposition \ref{t13}
and the Local Duality Theorem \ref{th8}.
\end{proof}
Let $R$ be a Cohen-Macaulay local ring with canonical module
$\omega_R$ and let $M$ be an $R$--module of finite projective
dimension. In \cite[Theorem 3.3(i)]{Ka}, Kawasaki proved that $M$ is
Cohen-Macaulay if and only if $M\underset{R}{\otimes}\omega_R$ is.
In \cite[Theorem 1.11]{KY}, Khatami and Yassemi generalized this
result for modules of finite Gorenstein dimension. We are going to
extend these results for modules in the Auslander class with respect
to a semidualizing module. These modules were defined by Avramov and
Foxby in \cite{F} and \cite{AvF}.
\begin{dfn}\label{def1}
Let $C$ be a semidualizing $R$--module. The \emph{Auslander class
with respect to} $C$, denoted $\mathcal{A}_{C}$, consists of all
$R$--modules $M$ satisfying the following conditions.
\begin{itemize}
\item[(i)] The natural map $\mu:M\longrightarrow\Hom_R(C,M\otimes_RC)$ is an isomorphism;
\item[(ii)] $\Tor_i^R(M,C)=0=\Ext^i_R(C,M\otimes_RC)$ for all $i>0$.
\end{itemize}
\end{dfn}
In the following we collect some properties and examples of modules
in the Auslander class with respect to a semidualizing module which
will be used in the rest of this paper.
\begin{eg}\label{example1}
\begin{enumerate}[(i)]
  \item{ If any two $R$-modules in a short exact sequence are
     in $\mathcal{A}_C$, then so is the third one \cite[Lemma 1.3]{F}.
     Hence, every module of finite projective dimension
     is in the Auslander class $\mathcal{A}_C$.}
  \item{ Over a Cohen-Macaulay local ring $R$ with canonical module $\omega_R$,
        $M\in\mathcal{A}_{\omega_R}$ if and only if $\gd_R(M)<\infty$
        \cite[Theorem 1]{F1}.}
   \item{The $\mathcal{I}_C$-injective dimension of $M$, denoted
    $\mathcal{I}_{C}$-$\id_R(M)$, is less than or equal to $n$ if and
           only if there is an exact sequence
            $$0\rightarrow M\rightarrow\Hom_R(C,I^0)\rightarrow\cdots\rightarrow\Hom_R(C,I^n)\rightarrow0,$$
             such that each $I^i$ is an injective $R$--module \cite[Corollary 2.10]{TW}.
              Note that if $M$ has a finite $\mathcal{I}_C$-injective
              dimension, then $M\in\mathcal{A}_C$
               \cite[Corollary 2.9]{TW}.}
\end{enumerate}
\end{eg}

Now,we generalize the results \cite[Theorem 3.3(i)]{Ka} and \cite[Theorem 1.11]{KY}.
\begin{lem}\label{lem2}
Let $R$ be a local ring, $C$ a semidualizing $R$--module, $n\geq 0$
an integer, and $M$ an $R$--module. If $M\in\mathcal{A}_C$, then the
following statements hold true.
\begin{enumerate}[(i)]
\item $\depth_R(M)=\depth_R(M\otimes_RC)$ and
$\dim_R(M)=\dim_R(M\otimes_RC)$;
\item $M$ satisfies $(S_n)$ if and only if $M\otimes_RC$ does;
\item $M$ is Cohen-Macaulay if and only if $M\otimes_RC$ is Cohen-Macaulay.
\end{enumerate}
\end{lem}
\begin{proof}
As $M\in\mathcal{A}_C$, $M\cong\Hom_R(C,M\otimes_RC)$ and
$\Ext^i_R(C,M\otimes_RC)=0$ for all $i>0$. It follows from
\cite[Lemma 4.1]{AY} that
$$\depth_{R}(M)=\depth_{R}(\Hom_{R}(C,M\otimes_{R}C))=
\depth_{R}(M\otimes_{R}C).$$ On the other hand, the fact that
$\Supp_{R}(C)=\Spec(R)$ implies
\[\begin{array}{rl}
\Ass_{R}(M)&=\Ass_{R}(\Hom_{R}(C,M\otimes_{R}C))\\
&=\Ass_{R}(M\otimes_{R}C).
\end{array}\] Therefore,
$\dim_{R}(M)=\dim_{R}(M\otimes_{R}C)$. Note that
$M_{\fp}\in\mathcal{A}_{C_\fp}$ for all $\fp\in\Supp_R(M)$ and so
the assertion is clear.
\end{proof}
In the following, we generalize \cite[Theorem 4.25]{AB} for modules
in the Auslander class with respect to a semidualizing module.
\begin{thm}\label{th5}
Let $C$ be a semidualizing $R$--module and $M$ an $R$--module.
Assume that $M\in\mathcal{A}_C$ and that $n$ is a positive integer.
Consider the following statements.
\begin{enumerate}[(i)]
\item{$\Ext^i_R(\Tr M,R)=0$ for $1\leq i\leq n$;}
\item{$\Ext^i_R(\Tr M,C)=0$ for $1\leq i\leq n$;}
\item{$M\otimes_R C$ satisfies $\widetilde{S}_n$;}
\item{$M$ satisfies $\widetilde{S}_n$.}
\end{enumerate}
Then we have the following
\begin{enumerate}[(a)]
\item{(i)$\Rightarrow$ (ii) $\Rightarrow$ (iii) $\Leftrightarrow$ (iv);}
\item{ If $\gd_{R_\fp}(M_\fp)<\infty$ for all
$\fp\in\X^{n-1}(R)$(e.g. $\id_{R_\fp}(C_\fp)<\infty$ for all
$\fp\in\X^{n-1}(R)$), then all the statements (i)-(iv) are
equivalent.}
\end{enumerate}
\end{thm}
\begin{proof}
(a) The equivalence of (iii) and (iv) follows from Lemma \ref{lem2}.

(i)$\Rightarrow$(ii). By \cite[Theorem 2.8]{AB}, there exists the
following exact sequence
\begin{equation}\tag{\ref{th5}.1}
\Tor_2^R(\trr_i(\Tr M),C)\rightarrow\Ext^i_R(\Tr
M,R)\otimes_RC\rightarrow \Ext^i_R(\Tr
M,C)\rightarrow\Tor_1^R(\mathcal{T}_{i+1}(\Tr M),C)\rightarrow0,
\end{equation}
 for
all $i>0$. As $\Ext^i_R(\Tr M,R)=0$ for $1\leq i\leq n$, by Remark
\ref{remark3}(ii),
\begin{equation}\tag{\ref{th5}.2}
\mathcal{T}_i(\Tr M)\approx\Omega\mathcal{T}_{i+1}(\Tr M) \text{ for
} 1\leq i\leq n.
\end{equation}
Note that $\mathcal{T}_1(\Tr M)\approx M$. As $M\in\mathcal{A}_C$,
it follows from (\ref{th5}.2) and Example \ref{example1}(i) that
$\trr_i(\Tr M)\in\mathcal{A}_C$ for all $1\leq i\leq n+1$ and so
$\Tor_j^R(\trr_i(\Tr M),C)=0$ for all $j>0$ and all $i$, $1\leq
i\leq n+1$. It follows from the exact sequence (\ref{th5}.1)
 that
$\Ext^i_R(\Tr M,C)=0$ for all $1\leq i\leq n$.

(ii)$\Rightarrow$(iii) This follows from Proposition \ref{t13}.

(b)(iv)$\Rightarrow$(i) Follows from Proposition \ref{t1}, by
replacing $C$ by $R$.
\end{proof}
Note that every module of finite projective dimension is in the
Auslander class with respect to $C$. The following result is an
immediate consequence of Theorem \ref{th5}.
\begin{cor}
Let $C$ be a semidualizing $R$--module and $M$ an $R$--module of
finite projective dimension. For a positive integer $n$, the
following are equivalent.
\begin{enumerate}[(i)]
\item{$M$ satisfies $\widetilde{S}_n$;}
\item{$M\otimes_R C$ satisfies $\widetilde{S}_n$ for every semidualizing module $C$;}
\item{$\Ext^i_R(\Tr M,C)=0$ for $1\leq i\leq n$ and for every semidualizing module $C$;}
\item{$M\otimes_R C$ satisfies $\widetilde{S}_n$ for some semidualizing module $C$;}
\item{$\Ext^i_R(\Tr M,C)=0$ for $1\leq i\leq n$ and for some semidualizing module $C$;}
\end{enumerate}
\end{cor}
\begin{cor}\label{cor7}
Let $R$ be a semiperfect ring, $C$ a semidualizing $R$--module and
$M$ a stable $R$--module. Assume that $M\in\mathcal{A}_C$ and that
$n$ is a positive integer. If $\gd_{R_\fp}(M_\fp)<\infty$ for all
$\fp\in\X^{n-1}(R)$(e.g. $\id_{R_\fp}(C_\fp)<\infty$ for all
$\fp\in\X^{n-1}(R)$), then the following are equivalent.
\begin{enumerate}[(i)]
\item{$M$ satisfies $\widetilde{S}_n$;}
\item{$M$ is horizontally linked and $\Ext^i_R(\lambda M,C)=0$ for
$0<i<n$.}
\end{enumerate}
\end{cor}
\begin{proof}
This is clear by Theorem \ref{th5} and Theorem \ref{MS}.
\end{proof}

\section{Linkage for Cohen-Macaulay modules}
The first main theorem in the theory of linkage was due to C.
Peskine and L. Szpiro. They proved that over a Gorenstein local ring
$R$ with linked ideals $\fa$ and $\fb$, $R/\fa$ is Cohen-Macaulay if
and only if $R/\fb$ is. They also gave a counter-example to show
that the above statement is no longer true if the base ring is
Cohen-Macaulay but non-Gorenstein. Attempts to generalize this
theorem lead to several development in linkage theory, especially by
C. Huneke and B. Ulrich (\cite{Hu} and \cite{HuUl}). In the theory
of linkage of modules, Martsinkovsky and Strooker generalize Peskine
and Szpiro's result for stable modules over Gorenstein local rings
\cite[Proposition 8]{MS}. In this section, we study the relation
between the Cohen-Macaulayness of $M\otimes_R\omega_R$ and $\lambda
M$, for a horizontally linked module $M$ over Cohen-Macaulay local
ring $R$ with canonical module $\omega_R$. We also present a
characterization of a maximal Cohen-Macaulay module (mCM) whose
linked module is also mCM.

\begin{thm} \label{theorem1} Let $R$ be a Cohen-Macaulay local ring of dimension $d$ with canonical module $\omega_R$, $M$ a
 horizontally linked $R$--module. If $M\otimes_R\omega_R$ satisfies $\widetilde{S}_1$, then the
following statements are equivalent.
\begin{itemize}
\item[(i)] $M\otimes_R\omega_R$ is maximal Cohen-Macaulay;
\item[(ii)] $\lambda M$ is maximal Cohen-Macaulay;
\item[(iii)] $M\otimes_R\omega_R$ satisfies $(S_n)$ for some $n$, $n>d-\depth_R(\lambda M)$;
\item[(iv)]$\lambda M$ satisfies $(S_n)$ for some $n$, $n>d-\depth_R(M\otimes_R\omega_R)$.
\end{itemize}
\end{thm}
\begin{proof}
The equivalence of (i) and (ii) is clear by Corollary \ref{c2}.
Trivially (i) implies (iii) and (iv).

(iii)$\Rightarrow$(ii). It follows from Corollary \ref{c2} that
$\hh^i_\fm(\lambda M)=0$ for all $i$, $d-n<i<d$. On the other hand,
$\hh^i_\fm(\lambda M)=0$ for all $i\leq d-n<\depth_R(\lambda M)$ and
so $\lambda M$ is maximal Cohen-Macaulay.

(iv)$\Rightarrow$(i). It follows from Proposition \ref{p3} that
$\hh^i_\fm(M\underset{R}{\otimes}\omega_R)=0$ for all $i$,
$d-n<i<d$. As $\depth_R(M\underset{R}{\otimes}\omega_R)>d-n$, we
conclude that $M\underset{R}{\otimes}\omega_R$ is maximal
Cohen-Macaulay.
\end{proof}
Let $R$ be a Cohen-Macaulay local ring with canonical module
$\omega_R$. If $R$ is generically Gorenstein, then $\omega_R$ can be
identified with an ideal of $R$. For any such identification
$\omega_R$ is an ideal of height one or equals $R$ (see
\cite[Proposition 3.3.18]{BH}).
\begin{thm}\label{the1}
Let $R$ be a non-Gorenstein Cohen-Macaulay local ring of dimension
$d$ which admits a canonical module $\omega_R$. Suppose that $R$ is
generically Gorenstein and that $M$ is maximal Cohen-Macaulay
horizontally linked $R$--module such that
$M\underset{R}{\otimes}\omega_R$ satisfies $(S_1)$. Then the
following statements are equivalent.
\begin{itemize}
\item[(i)]{$\lambda M$ is maximal Cohen-Macaulay.}
\item[(ii)]{$M/\omega_RM$ is Cohen-Macaulay of dimension $d-1$.}
\end{itemize}
\end{thm}
\begin{proof}
As $R$ is generically Gorenstein and it is not Gorenstein,
$\omega_R$ can be identified with an ideal of height one. The exact
sequence $0\rightarrow\omega_R\rightarrow R\rightarrow
R/\omega_R\rightarrow0$ implies the exact sequence
$$0\rightarrow\Tor_1^R(M,R/\omega_R)\rightarrow M\underset{R}{\otimes}\omega_R\rightarrow M\rightarrow M\underset{R}{\otimes} R/\omega_R\rightarrow0.$$
As $M\underset{R}{\otimes}\omega_R$ satisfies $(S_1)$ and $R$ is
generically Gorenstein, it follows that $\Tor_1^R(M,R/\omega_R)=0$
and one has the following exact sequence
\begin{equation}\tag{\ref{the1}.1}
0\rightarrow M\underset{R}{\otimes}\omega_R\rightarrow M\rightarrow
M\underset{R}{\otimes} R/\omega_R\rightarrow0.
\end{equation}
As $M$ is maximal Cohen-Macaulay, from the exact sequence
(\ref{the1}.1), it is clear that $M\underset{R}{\otimes}\omega_R$ is
maximal Cohen-Macaulay if and only if $M\underset{R}{\otimes}
R/\omega_R$ is Cohen-Macaulay of dimension $d-1$. Now the assertion
is clear by Theorem \ref{theorem1}.
\end{proof}

As an immediate consequence, we have the following result.
\begin{cor}\label{theorem3} Let $R$ be a Cohen-Macaulay local ring of dimension $d$ with canonical module $\omega_R$
which is not Gorenstein but it is generically Gorenstein. If the
ideals $I$ and $J$ are linked by zero ideal such that
$I\omega_R=I\cap\omega_R$ and $R/I$ is Cohen-Macaulay, then $R/J$ is
Cohen-Macaulay if and only if $R/I+\omega_R$ is Cohen-Macaulay of
dimension $d-1$.
\end{cor}
\begin{proof} Note that $\omega_R$ is identified with an ideal of $R$ \cite[Proposition 3.3.18]{BH}.
As $I\omega_R=I\cap\omega_R$, $\Tor_1^R(R/\omega_R,
R/I)\cong\frac{I\cap\omega_R}{I\omega_R}=0$. By similar argument as
in the proof of Theorem \ref{the1}, we obtain the exact sequence
$0\rightarrow (R/I)\underset{R}{\otimes}\omega_R\rightarrow R/I$. As
$R/I$ is horizontally linked, it is a first syzygy module. Therefore
$(R/I)\underset{R}{\otimes}\omega_R$ is a first syzygy and so
satisfies $(S_1)$. Now the assertion is clear by Theorem \ref{the1}.
\end{proof}

\section{Linkage of modules of finite $\gc$--dimensions}
Throughout the rest of the paper, $R$ is semiperfect ring and $C$ is
a semidualizing $R$--module. We study the theory of linkage of
modules which have finite $\gc$--dimensions. The connection of the
Serre condition $(S_n)$ on a horizontally linked $R$--module of
finite $\gc$--dimension with the vanishing of certain cohomology
modules of its linked module is discussed.

An $R$--module $M$ is said to be \emph{linked} to an $R$--module
$N$, by an ideal $\fc$ of $R$, if
$\fc\subseteq\ann_R(M)\cap\ann_R(N)$ and $M$ and $N$ are
horizontally linked as $R/{\fc}$--modules. In this situation we
denote $M\underset{\fc}{\thicksim}N$ \cite[Definition 4]{MS}.

Recall that, for an $R$--module $M$ we always have
$\gr_R(M)\leq\gkd_R(M)$. The module $M$ is called
$\gc$-\emph{perfect} if $\gr_R(M)=\gkd_R(M)$. An $R$--module $M$ is
called $\gc$-\emph{Gorenstein} if it is $\gc$-perfect and
$\Ext^n_R(M,C)$ is cyclic, where $n=\gkd_R(M)$. An ideal $I$ is
called $\gc$-perfect(resp.$\gc$-Gorenstein) if $R/I$ is
$\gc$-perfect(resp.$\gc$-Gorenstein) as $R$--module. Note that if
$C$ is a semidualizing $R$--module and $I$ is a $\gc$-Gorenstein
ideal of $\gc$-dimension $n$, then $\Ext^n_R(R/I,C)\cong R/I$ (see
\cite[10]{G}).

We recall a result of Golod to be used in the following and more in
the sequel \cite[Proposition 5]{G}.
\begin{thm}\label{G2}
Let $R$ be a local ring, $I$ a $\gc$-perfect ideal and $C$ a
semidualizing $R$--module. Set $K=\Ext^{\tiny{\gr(I)}}_R(R/I,C)$.
Then the following statements hold true.
\begin{itemize}
        \item[(i)]{$K$ is a semidualizing $R/I$--module.}
         \item[(ii)]If $M$ is a $R/I$--module, then $\gkd_R(M)<\infty$ if and only if $\gkkd_{R/I}(M)<\infty$, and
         if these dimensions are finite then
$\gkd_R(M)=\gr(I)+\gkkd_{R/I}(M)$.
\end{itemize}
\end{thm}

We first present a generalization of \cite[Theorem 4.1]{Sc} for
modules of finite $\gc$-dimension.
\begin{prop}
Let $R$ be a Cohen-Macaulay local ring of dimension $d$. Suppose
that $M$ is an $R$--module of finite $\gc$-dimension and that $\fc$
is a $\gc$-Gorenstein ideal of $R$. Assume that $M$ is linked by
$\fc$ and that $n$ is a positive integer $n$. Then the following
statements are equivalent.
\begin{itemize}
\item[(i)]{$M$ satisfies $(S_n)$.}
\item[(ii)]{$\hh^i_\fm(\lambda_{R/\fc} M)=0$ for all $i$, $\dim R/\fc-n<i<\dim R/\fc$.}
\end{itemize}
\end{prop}
\begin{proof}
As $\fc$ is $\gc$-Gorenstein and $\gkd_R(M)<\infty$, it follows from
Theorem \ref{G2} that $\gd_{R/\fc}(M)<\infty$. Note that $R/\fc$ is
a Cohen-Macaulay ring. By \cite[Theorem 4.2]{DS}, $M$ satisfies
$(S_n)$ if and only if $\hh^i_{\fm}(\lambda_{R/\fc}
M)\cong\hh^i_{\fm/\fc}(\lambda_{R/\fc} M)=0$ for all $i$, $\dim
R/\fc-n<i<\dim R/\fc$.
\end{proof}
The reduced grade of a module $M$ with respect to a semidualizing
$C$ defined as follows
$$\rgr(M,C)=\inf\{i>0\mid \Ext^i_R(M,C)\neq0\}.$$
We denote by $\rgr(M)$, the reduced grade of $M$ with respect to
$R$. If $\Ext^i_R(M,C)=0$ for all $i>0$, then $\rgr(M,C)=+\infty$.
Note that if $M$ has a finite and positive $\gc$-dimension, then
$\rgr(M,C)\leq\gkd_R(M)$ by Theorem \ref{G3}.

Recall that two $R$-modules $M$ and $N$ are said to be in the same
even linkage class, or evenly linked, if there is a chain of even
length of linked modules that starts with $M$ and ends with $N$. The
following result, shows that the condition $\widetilde{S}_n$ is
preserved under even linkage.
\begin{thm}\label{prop}
Let $R$ be a local ring, $C$ a semidualizing $R$--module and
$\fc_1$, $\fc_2$ two $\gc$-Gorenstein ideals. Suppose that $M_1$,
$M$, and $M_2$ are $R$--modules such that
$M_1\underset{\fc_1}{\thicksim}M$ and
$M\underset{\fc_2}{\thicksim}M_2$. Assume that $\gkd_R(M)<\infty$
and that $n>0$ is an integer. Then $M_1$ satisfies $\widetilde{S}_n$
if and only if $M_2$ satisfies $\widetilde{S}_n$. In particular, if
$R$ is Cohen-Macaulay then $M_1$ is Cohen-Macaulay if and only if
$M_2$ is.
\end{thm}
\begin{proof}
As $\fc_1$ and $\fc_2$ are $\gc$-Gorenstein ideals and
$\gkd_R(M)<\infty$ we have $\gd_{R/\fc_1}(M)<\infty$ and
$\gd_{R/\fc_2}(M)<\infty$ by Theorem \ref{G2}. Note that by
\cite[Lemma 5.8]{DS}, we denote the common value of $\gr_R(\fc_1)$
and $\gr_R(\fc_2)$ by $k$. By \cite[Corollary]{G},
\begin{equation}\tag{\ref{prop}.1}
\Ext^i_{R/\fc_1}(M,R/\fc_1)\cong\Ext^{i+k}_{R}(M,C)\cong\Ext^i_{R/\fc_2}(M,R/\fc_2),
\end{equation}
for all $i>0$. By (\ref{prop}.1),
$\rgr_{R/\fc_1}(M)=\rgr_{R/\fc_2}(M)$. By \cite[Proposition
2.6]{DS}, $M_1=\lambda_{R/\fc_1}M$ satisfies $\widetilde{S}_n$ if
and only if $\rgr_{R/\fc_1}(M)=\rgr_{R/\fc_2}(M)\geq n$, and this is
equivalent to saying that $M_2=\lambda_{R/\fc_2}M$ satisfies
$\widetilde{S}_n$ by using \cite[Proposition 2.6]{DS} again.
\end{proof}
In the following, we express the associated primes of the
$\Ext^{\tiny{\rgr(M,C)}}_R(M,C)$ for a horizontally linked module
$M$ of finite and positive $\gc$-dimension in terms of $\lambda M$,
which is a generalization of \cite[Lemma 2.1]{DS}. For an integer
$n>0$, we denote the compositions
$\mathcal{T}_n^{C}:=\trk\Omega^{n-1}$.
\begin{lem}\label{lem4}
Let $M$ be a horizontally-linked $R$--module of finite and positive
$\gc$-dimension. Set $n=\rgr_R(M,C)$. If $\lambda M\in\mathcal{A}_C$
(e.g. $\pd_R(\lambda M)<\infty$), then
$$\Ass_R(\Ext^n_R(M,C))=\{\fp\in\Spec R\mid \gcpd(M_\fp)\neq0, \depth_{R_\fp}((\lambda M)_\fp)=n=\rgr_{R_\fp}(M_\fp,C_\fp)\}.$$
\end{lem}
\begin{proof} First note that $\lambda M$ is a first syzygy of $\Tr M$ and so $\Tr M\in\mathcal{A}_C$ by Example \ref{example1}(i).
Hence by Lemma \ref{lem2} and Remark \ref{remark3}(i)
\begin{equation}\tag{\ref{lem4}.1}
\depth_{R_\fp}((\Tr M)_\fp)=\depth_{R_\fp}((\Tr
M)_\fp\otimes_{R_\fp}C_\fp)=\depth_{R_\fp}((\trk M)_\fp),
\end{equation}
for all $\fp\in\Spec R$.
 Let $\cdots\rightarrow P_i\rightarrow\cdots\rightarrow
P_0\rightarrow M\rightarrow0$ be the minimal projective resolution
of $M$. As $\Ext^i_R(M,C)=0$ for all $0<i<n$, applying functor
$(-)^{\triangledown}=\Hom_R(-,C)$ on the minimal projective
resolution of $M$ implies the following exact sequences:
\begin{equation}\tag{\ref{lem4}.2}
0\rightarrow\trk M\rightarrow
(P_2)^{\triangledown}\rightarrow\cdots\rightarrow(P_n)^{\triangledown}\rightarrow\trc_{n}M\rightarrow0.
\end{equation}
\begin{equation}\tag{\ref{lem4}.3}
0\rightarrow\Ext^n_R(M,C)\rightarrow\trc_nM\rightarrow L\rightarrow0
\end{equation}
\begin{equation}\tag{\ref{lem4}.4}
0\rightarrow L\rightarrow
\overset{m}{\oplus}C\rightarrow\trc_{n+1}M\rightarrow0
\end{equation}
By Theorem \ref{MS}, $M$ is a first syzygy. Therefore, if
$\gcpd(M_\fp)\neq0$ for some $\fp\in\Spec R$, then it follows from
Theorem \ref{G3} that
\[\begin{array}{rl}\tag{\ref{lem4}.5}
n\leq\rgr(M_\fp,C_\fp)\leq\gcpd(M_\fp)\\
=\depth R_\fp-\depth_{R_\fp}( M_\fp)<\depth R_\fp.
\end{array}\]
Assume that $\fp\in\Ass_R(\Ext^n_R(M,C))$ so that
$\rgr_{R_\fp}(M_\fp,C_\fp)=n$ and $\gcpd(M_\fp)\neq0$ by Theorem
\ref{G3}. It follows from the exact sequence (\ref{lem4}.3) that
$\depth_{R_\fp}((\trc_{n} M)_\fp)=0$. Note that
$\depth_{R_\fp}(((P_i)^{\triangledown})_\fp)=\depth_{R_\fp}(C_\fp)=\depth
R_\fp$ for all $i$.  By localizing the exact sequence (\ref{lem4}.2)
at $\fp$, we conclude from (\ref{lem4}.5) that $\depth_{R_\fp}((\trk
M)_\fp)=n-1$. Hence, by (\ref{lem4}.1) and (\ref{lem4}.5),
$\depth_{R_\fp}((\lambda M)_\fp)=n$.

Now assume that $\fp\in\Spec R$ such that $\gcpd(M_\fp)\neq0$ and
$\depth_R((\lambda M)_\fp)=n=\rgr_{R_\fp}(M_\fp,C_\fp)$. Hence
$\depth_{R_\fp}((\trk M)_\fp)=n-1$ by (\ref{lem4}.1) and
(\ref{lem4}.5). By localizing the exact sequence (\ref{lem4}.2) at
$\fp$, we conclude from (\ref{lem4}.5) that
$\depth_{R_\fp}((\trc_{n}M)_\fp)=0$. As
$\depth_{R_\fp}(C_\fp)=\depth R_\fp>0$, we conclude from the exact
sequence (\ref{lem4}.4) that $\depth_{R_\fp}(L_\fp)>0$. It follows
from the exact sequence (\ref{lem4}.3) that
$\depth_{R_\fp}(\Ext^n_R(M,C)_\fp)=0$. In other words,
$\fp\in\Ass_R(\Ext^n_R(M,C))$.
\end{proof}
In order to describe the next result, we recall some notations from
the literature. An $R$--module $S$ is called secondary if, for any
$r\in R$, the multiplication map $S\overset{r.}{\longrightarrow} S$
is either surjective or nilpotent; in this case $\fp:=\sqrt{0:_RS}$
is a prime ideal and $S$ is called a $\fp$-secondary module. It is
well-known that any Artinian module $T$ is minimally representable
as a sum of its secondary submodules $T=S_1+\cdots + S_n$ such that
$S_i$ is a $\fp_i$--secondary, $1\leq i\leq n$, for some distinct
prime ideals $\fp_1, \cdots, \fp_n$. These primes are said to be
attached and are denoted by $\Att_R(T)$ \cite{Mc}. We denote the
non-Cohen-Macaulay locus of $M$ by
$$\cm(M):=\{\fp\in\Spec R\mid M_\fp \text{ is not a Cohen-Macaulay }
R_\fp\text{-module}\}.$$ For a non--Cohen-Macaulay module $M$ of
dimension $d$ over a local ring $(R, \fm)$, set
$$\cc(M)=\sup\{i<d \mid\hh^i_\fm(M)\neq0\}.$$

As an application, one can describe the attached prime ideals of
$\hh^{\tiny{\cc(M)}}_\fm(M)$ for a non-Cohen-Macaulay horizontally
linked $R$--module $M$.
\begin{cor}\label{cor2}
Let $R$ be a Cohen-Macaulay local ring of dimension $d$ with
canonical module $\omega_R$ and let $M$ be an $R$--module which is a
horizontally linked non-Cohen-Macaulay. If $\gd_R(\lambda M)<\infty$
(e.g. $R$ is Gorenstein), then
$$\Att_R(\hh^{\tiny{\cc(M)}}_\fm(M))=\{\fp\in\Spec R\mid \fp\in\cm(M), \depth_{R_\fp}((\lambda M)_\fp)=d-\cc(M)=\dim R_\fp-\cc(M_\fp)\}.$$
\end{cor}
\begin{proof}
As $M$ is horizontally linked, it is a first syzygy and so $\dim
M_\fp=\dim R_\fp$ for all $\fp\in \Supp(M)$. By \cite[3.4]{Sh}, and
the Local Duality Theorem \ref{th8},
$$\Att_R(\hh^{\tiny{\cc(M)}}_\fm(M))=\Ass_R(\Ext^{\tiny{\rgr(M,\omega_R)}}_R(M,\omega_R)).$$
By Example \ref{example1}(ii), $\lambda M\in\mathcal{A}_{\omega_R}$.
Now the assertion is clear by Lemma \ref{lem4} and the Local Duality
Theorem.
\end{proof}
In the following result, we investigate the relation between the
Serre condition $\widetilde{S}_n$ on a horizontally linked module
with the vanishing of certain cohomology modules of its linked
module.
\begin{thm}\label{th1}
Let $M$ be a stable $R$--module of finite $\gc$-dimension and
$\lambda M\in\mathcal{A}_C$ (e.g. $\pd_R(\lambda M)<\infty$). For an
integer $n>0$, the following statements hold true.
\begin{enumerate}[(i)]
           \item{$M$ satisfies $\widetilde{S}_n$ if and only if $M$ is horizontally linked and $\rgr_R(\lambda M)\geq n$;}
           \item{If $M$ is horizontally linked, then $\rgr(M,C)\geq
           n$ if and only if $\lambda M$ satisfies
           $\widetilde{S}_n$.}
\end{enumerate}
\end{thm}
\begin{proof}
(i). By Proposition \ref{t1}, $M$ satisfies $\widetilde{S}_n$ if and
only if $\Ext^i_R(\trk M,C)=0$ for all $1\leq i\leq n$. As mentioned
in Remark \ref{remark3}(i), $\trk M\cong\Tr M\otimes_RC$. Similar to
the proof of \cite[Lemma 2.2]{AIT}, it follows from adjointness
isomorphism
 $$\R\Hom_R(\Tr M\otimes^{\emph{L}} C, C)\cong\R\Hom_R(\Tr M, \R\Hom_R(C, C))\cong\R\Hom_R(\Tr M, R),$$
 in the derived category of $R$, that there is a third quadrant spectral sequence
$$\E^{p,q}_2=\Ext^p_R(\Tor_q^R(\Tr M,C),C)\Rightarrow\Ext^{p+q}_R(\Tr M,R).$$  By Example \ref{example1}(i), $\Tr M\in\mathcal{A}_C$ and so
$\Tor_i^R(\Tr M,C)=0$ for all $i>0$. Hence $\E^{p,q}_2=0$ for all
$q>0$. Therefore, the spectral sequence collapses on $p$-axis and so
$$\Ext^i_R(\trk M,C)\cong\Ext^i_R(\Tr M\otimes_RC,C)\cong\Ext^i_R(\Tr
M,R),$$ for all $i\geq0$. Therefore, $M$ satisfies $\widetilde{S}_n$
if and only if $\Ext^i_R(\Tr M,R)=0$ for all $1\leq i\leq n$. Now
the assertion is clear by Theorem \ref{MS}.

(ii). Let $\cdots\rightarrow P_i\rightarrow\cdots\rightarrow
P_0\rightarrow M\rightarrow0$ be the minimal projective resolution
of $M$ and let $\Ext^i_R(M,C)=0$ for all $0<i<n$. Applying the
functor $(-)^{\triangledown}=\Hom_R(-,C)$ on the minimal projective
resolution of $M$, gives the exact sequence
$$0\rightarrow\trk M\rightarrow
(P_2)^{\triangledown}\rightarrow\cdots\rightarrow(P_n)^{\triangledown}.$$
Now it is easy to see that $\trk M(\cong\Tr M\otimes_R C)$ satisfies
$\widetilde{S}_{n-1}$. As $\Tr M \in\mathcal{A}_C$, by Lemma
\ref{lem2},$\Tr M$ satisfies $\widetilde{S}_{n-1}$. Therefore,
$\lambda M$ satisfies $\widetilde{S}_n$.

Conversely, assume that $\lambda M$ satisfies $\widetilde{S}_n$. If
$\gkd_R(M)=0$ then we have nothing to prove. Assume that
$\gkd_R(M)>0$. Set $k=\rgr(M,C)$ and suppose that
$\fp\in\Ass_R(\Ext^k_R(M,C))$. By Lemma \ref{lem4}, we have
$k=\depth_{R_\fp}((\lambda M)_\fp)$. Also we have
$k=\rgr_{R_\fp}(M_\fp,C_\fp)\leq\gcpd(M_\fp)<\depth R_\fp$. As
$\lambda M$ satisfies $\widetilde{S}_n$, it follows
$k=\depth_{R_\fp}((\lambda M)_\fp)\geq\min\{n,\depth R_\fp\}$ and so
$\depth R_\fp>n$. Therefore $\rgr(M,C)\geq n$.
\end{proof}
Let $R$ be a Gorenstein local ring and $M$ a horizontally linked
$R$--module. In \cite[Proposition 8]{MS}, it is shown that the
maximal Cohen-Macaulayness of $M$ and $\lambda M$ are equivalent. As
a consequence of Theorem \ref{th1}, we can generalize
\cite[Proposition 8]{MS} as follows.
\begin{cor}\label{cor5}
Let $R$ be a Cohen-Macaulay local ring of dimension $d$ and $M$ a
stable $R$--module of finite $\gc$-dimension. If $\lambda
M\in\mathcal{A}_C$ (e.g. $\mathcal{I}_{C}$-$\id_R(\lambda
M)<\infty$), then the following statements are equivalent.
\begin{enumerate}[(i)]
\item{$M$ is maximal Cohen-Macaulay;}
\item{$\lambda M$ is maximal Cohen-Macaulay and $M$ is horizontally linked;}
\item{$\lambda M$ satisfies ($S_n$) for some $n>d-\depth_R(M)$ and $M$ is horizontally linked.}
\end{enumerate}
\end{cor}
\begin{proof}
(i)$\Rightarrow$(ii). It follows from Theorem \ref{G3} and Remark
\ref{rem1} that $\trk M$ is maximal Cohen-Macaulay. As mentioned in
Remark \ref{remark3}(i), $\Tr M\otimes_RC\cong\trk M$. By Example
\ref{example1}(i), $\Tr M\in\mathcal{A}_C$. Now it follows from
Lemma \ref{lem2} that $\Tr M$ is maximal Cohen-Macaulay. Therefore
$\lambda M$ is maximal Cohen-Macaulay. As we have seen in the proof
of Theorem \ref{th1}, $\Ext^i_R(\trk M,C)\cong\Ext^i_R(\Tr M,R)$ for
all $i>0$. As $\gkd_R(M)=0$, $\Ext^i_R(\trk M,C)=0$ for all $i>0$.
Hence $\Ext^1_R(\Tr M,R)=0$ and so $M$ is horizontally linked by
Theorem \ref{MS}.

(ii)$\Rightarrow$(iii). Assume that $M$ is horizontally linked and
that $\lambda M$ is maximal Cohen-Macaulay so that it satisfies
($S_n$) for all $n$ and we have nothing to prove.

(iii)$\Rightarrow$(i). By Theorem \ref{th1}, $\rgr(M,C)\geq n$. Now
assume contrarily that $M$ is not maximal Cohen-Macaulay. Hence
$\gkd_R(M)>0$ and so we obtain the following inequality from Theorem
\ref{G3}:
$$n\leq\rgr(M,C)\leq\gkd_R(M)=d-\depth_R(M),$$
which is a contradiction.
\end{proof}
Let $\fa$ be a $\gc$-perfect ideal of $R$, i.e. $R/\fa$ is
$\gc$-perfect as $R$--module. By Theorem \ref{G2},
$\Ext^{\gr(\fa)}_R(R/\fa,C)$ is a semidualizing $R/\fa$--module. As
an immediate consequence of Corollary \ref{cor5}, we have the
following result:
\begin{cor}\label{cor6}
Let $R$ be a Cohen-Macaulay local ring, $M$ an $R$--module of finite
$\gc$-dimension linked by a $\gc$-perfect ideal $\fa$. Assume that
$\lambda_{R/\fa}M\in\mathcal{A}_K$, where
$K:=\Ext^{\tiny{\gr(\fa)}}_R(R/\fa,C)$. Then $M$ is Cohen-Macaulay
if and only if $\lambda_{R/\fa}M$ is so.
\end{cor}
\begin{proof}
First note that $R/\fa$ is a Cohen-Macaulay local ring. As $M$ is
linked by the ideal $\fa$, it follows from Theorem \ref{MS} that $M$
is a first syzygy as an $R/\fa$--module and so $\dim_{R/\fa}(M)=\dim
R/\fa=\dim_R M$. Therefore $M$ is Cohen-Macaulay if and only if it
is maximal Cohen-Macaulay $R/\fa$--module. By Theorem \ref{G2},
$\gkkd_{R/\fa}(M)<\infty$. Now the assertion is clear by Corollary
\ref{cor5}.
\end{proof}
In the following result, we characterize when the local cohomology
group $\hh^{\cc(M)}_\fm(M)$ is of finite length in terms a numerical
equality and some inequalities, in certain cases.
\begin{thm}\label{cor3}
Let $R$ be a Cohen-Macaulay local ring of dimension $d$ with
canonical module $\omega_R$ and let $M$ be a horizontally linked non
Cohen-Macaulay $R$--module. If $\gd_R(\lambda M)<\infty$, then the
following are equivalent:
\begin{enumerate}[(i)]
\item $\hh^{\cc(M)}_\fm(M)$ is finitely generated;
\item{ $\depth_R(\lambda M)+\cc(M)=d$ and
$\depth_{R_\fp}((\lambda M)_\fp)+\cc(M)>d$ for all
$\fp\in\cm(M)\setminus\{\fm\}$.}
\end{enumerate}
\end{thm}
\begin{proof}
By \cite[Corollary 7.2.12]{BS}, $\hh^{\cc(M)}_\fm(M)$ is finitely
generated if and only if $\Att_R(\hh^{\cc(M)}_\fm(M))=\{\fm\}$. Now
the implication of $(ii)\Rightarrow(i)$ follows from Corollary
\ref{cor2}.\\
(i)$\Rightarrow$(ii). By Corollary \ref{cor2}, $\depth_R(\lambda
M)+\cc(M)=d$. Assume contrarily that
\begin{equation}\tag{\ref{cor3}.1}
\depth_{R_\fp}((\lambda M)_\fp)+\cc(M)\leq d  \text{ for some }
\fp\in\cm(M)\setminus\{\fm\}.
\end{equation}
Note that, by Local Duality Theorem \ref{th8},
$d-\cc(M)=\rgr(M,\omega_R)$. By Theorem \ref{th1}(ii), we obtain
\begin{equation}\tag{\ref{cor3}.2}
\rgr(M,\omega_R)\leq\rgr(M_\fp,\omega_{R_\fp})\leq\depth_{R_\fp}((\lambda
M)_\fp).
\end{equation}
It follows from (\ref{cor3}.1), (\ref{cor3}.2) that
$\rgr(M,\omega_R)=\rgr(M_\fp,\omega_{R_\fp})=\depth_{R_\fp}((\lambda
M)_\fp)$. Hence by Corollary \ref{cor2},
$\fp\in\Att_R(\hh^{\cc(M)}_\fm(M))=\{\fm\}$, which is a
contradiction.
\end{proof}
For a ring $R$, set $\X^i(R)=\{\fp\in\Spec R\mid\depth R_\fp\leq
i\}$. We have the following characterization of horizontally linked
module of $\gc$-dimension zero.
\begin{thm}\label{th2}
Let $M$ be a horizontally linked $R$--module of finite
$\gc$-dimension and $\lambda M\in\mathcal{A}_C$. Then $\gkd_R(M)=0$
if and only if $\depth_{R_\fp}(M_\fp)+\depth_{R_\fp}((\lambda
M)_\fp)>\depth R_\fp$ for all $\fp\in\Spec R\setminus\X^0(R)$.
\end{thm}
\begin{proof}
Set $n=\gkd_R(M)$. Let
$\depth_{R_\fp}(M_\fp)+\depth_{R_\fp}((\lambda M)_\fp)>\depth R_\fp$
for all $\fp\in\Spec R\setminus\X^0(R)$. Assume contrarily that,
$n>0$. By Lemma \ref{lem4}, there exists $\fp\in\Spec R$ such that
$\gcpd(M_\fp)\neq0$ and $\rgr(M,C)=\depth_{R_\fp}((\lambda M)_\fp)$.
Therefore by Theorem \ref{G3}, we have
\[\begin{array}{rl}
\rgr(M,C)>\depth R_\fp-\depth_{R_\fp}(M_\fp)\\
=\gcpd(M_\fp)\geq\rgr(M_\fp,C_\fp),
\end{array}\]
which is a contradiction. Therefore $\gkd_R(M)=0$.

Conversely assume that, $\gkd_R(M)=0$. Assume contrarily that
$\depth_{R_\fp}(M_\fp)+\depth_{R_\fp}((\lambda M)_\fp)\leq\depth
R_\fp$ for some $\fp\in\Spec R\setminus X^0(R)$. It follows that
$\fp\in\Supp_R(M)$. Note that $\gcpd(M_\fp)=0$. Hence $\gcpd(\trcp
M_\fp)=0$ and so $\depth_{R_\fp}(\trcp M_\fp)=\depth R_\fp$. By
Example \ref{example1}(i), $\Tr M\in\mathcal{A}_C$. It follows from
Remark \ref{remark3}(i) and Lemma \ref{lem2} that
$\depth_{R_\fp}(\trcp M_\fp)=\depth_{R_\fp}(\Tr M_\fp)$. Therefore
$\depth_{R_\fp}(\lambda M_\fp)\geq\depth R_\fp$ Thus we obtain
$\depth_{R_\fp}(M_\fp)=0$. By Theorem \ref{MS}, $M$ is a first
syzygy and so $\depth R_\fp=0$ which is a contradiction.
\end{proof}
Recall that a module $M$ is said to be \emph{horizontally
self-linked} if $M\cong\lambda M$ (see \cite[Definition 7]{MS}). It
is immediate that the above characterization becomes simpler
whenever $M$ is horizontally self-linked.
\begin{cor}\label{cor}
Let $M$ be a horizontally self-linked $R$--module of finite
$\gc$-dimension and $M\in\mathcal{A}_C$. Then $\gkd_R(M)=0$ if and
only if $\depth_{R_\fp}(M_\fp)
>\frac{1}{2}(\depth R_\fp)$ for all $\fp\in\Spec R\setminus\X^0(R)$.
\end{cor}
Let $R$ be a local ring. For an $R$--module $M$, we recall from
\cite{EGS} that $\syz(M)$ denotes the largest integer $n$ for which
$M$ is the $n$th syzygy in a minimal free resolution of an
$R$--module $N$. In general, for a horizontally linked $R$--module
$M$ of finite and positive $\gc$-dimension, one has
\begin{equation}\tag{\ref{cor}.1}\label{equ}
\rgr(\lambda M)\leq\syz(M)\leq\depth_R(M),
\end{equation}
by \cite[Proposition 11]{M1} and Theorem \ref{MS}. In the following,
we study under which conditions the equality holds in the above
inequality.

For an $R$--module $M$, set
$$\ng(M)=\{\fp\in\Spec(R)\mid\gcpd(M_{\fp})\neq0\}.$$  The following is a
generalization of \cite[Theorem 2.7]{DS}.
\begin{thm}\label{th3}
Let $R$ be a local ring, $M$ an $R$--module with
$0<\gkd_R(M)<\infty$ and $\lambda M\in\mathcal{A}_C$. If $M$ is
horizontally linked then the following conditions are equivalent.
\begin{enumerate}[(i)]
          \item {$\depth_R(M)=\syz(M)=\rgr(\lambda M)$;}
          \item {$\fm\in\Ass_R(\Ext^{\tiny{\rgr(\lambda M)}}_R(\lambda M,R))$;}
          \item {$\depth_R(M)\leq\depth_{R_{\fp}}(M_{\fp})$, for each $\fp\in\ng(M)$}.
\end{enumerate}
\end{thm}
\begin{proof}
We first note some general facts. Set $t=\rgr(\lambda M)$. As
$\gkd_R(M)>0$, by Theorem \ref{th1}(i), $t<\depth R$. By Remark
\ref{remark3}(ii), there exists the following exact sequence
\begin{equation}\tag{\ref{th3}.1}
0\rightarrow\Ext^t_R(\lambda M,R)\rightarrow\mathcal{T}_t(\lambda
M)\rightarrow\lambda^2\mathcal{T}_t(\lambda M)\rightarrow0,
\end{equation}
and also $\mathcal{T}_i(\lambda M)\cong\lambda^2\mathcal{T}_i\lambda
M\approx\Omega\mathcal{T}_{i+1}(\lambda M)$ for all $i$, $0<i<t$.
Therefore,
\begin{equation}\tag{\ref{th3}.2}
M\cong\lambda^2 M=\Omega\mathcal{T}_1(\lambda
M)\approx\Omega^2\mathcal{T}_2(\lambda
M)\approx\cdots\approx\Omega^t\mathcal{T}_t(\lambda M).
\end{equation}
$(i)\Rightarrow (ii)$ As $\depth_R(M)=t<\depth R$, it is easy to see
that $\depth_R(\mathcal{T}_t(\lambda M))=0$ by the (\ref{th3}.2).
From the exact sequence (\ref{th3}.1), as
$\depth_R(\lambda^2\mathcal{T}_t(\lambda M))>0$, we find that
$\depth_R(\Ext^t_R(\lambda M,R))=0$ and so
$\fm\in\Ass_R(\Ext^t_R(\lambda M,R))$.

$(ii)\Rightarrow (i)$ By the exact sequence (\ref{th3}.1), it is
obvious that $\depth_R(\mathcal{T}_t(\lambda M))=0$. As $t<\depth
R$, $\depth_R(M)=t$ by (\ref{th3}.2). Now the assertion is clear by
the inequality (\ref{equ}).

$(i)\Rightarrow(iii)$ By Theorem \ref{th1}(i), $M$ satisfies
$\widetilde{S}_t$. Hence if $\fp\in\ng(M)$ then
$\depth_{R_{\fp}}(M_{\fp})\geq t=\depth_R(M)$.

$(iii)\Rightarrow (i)$ Let $\fp\in\Ass_R(\Ext^{t}_R(\lambda M,R))$.
Therefore, $\fp
R_{\fp}\in\Ass_{R_{\fp}}(\Ext^t_{R_{\fp}}(\lambda_{R_{\fp}}M_{\fp},R_{\fp}))$
and $\rgr_{R_{\fp}}(\lambda_{R_{\fp}} M_{\fp})=t$. By the exact
sequence (\ref{th3}.1), we have
$\depth_{R_{\fp}}(\mathcal{T}_t(\lambda_{R_{\fp}}M_{\fp}))=0$. If
$\gcpd(M_{\fp})=0$, then $\gcpd((\trk M)_\fp)=0$ and so
$\Ext^i_R(\trk M,C)_\fp=0$ for all $i>0$. On the other hand, as seen
in the proof of Theorem \ref{th1}, we have $$\Ext^i_R(\Tr
M,R)\cong\Ext^i_R(\trk M,C) \text{ for all } i>0,$$ which leads to a
contradiction. Thus $\gcpd(M_{\fp})\neq0$. By Theorem \ref{th1}(i),
$M$ satisfies $\widetilde{S}_t$ which implies $t<\depth R_{\fp}$.
Now by localizing (\ref{th3}.2) at $\fp$, we conclude that
$\depth_{R_{\fp}}(M_{\fp})=t$. By our assumption, $\depth_R(M)\leq
t$. Now the assertion is clear by the inequality (\ref{equ}).
\end{proof}
We can express the reduced grade with respect to a semidualizing
module of a horizontally linked module in terms of the depth of
linked module as follows, which is a generalization of
\cite[Proposition 2.2]{DS}.
\begin{thm}\label{th6}
Let $M$ be a horizontally-linked $R$--module of finite
$\gc$-dimension and $\lambda M\in\mathcal{A}_C$. Then
$$\rgr(M,C)=\inf\{\depth_{R_\fp}\bigl((\lambda M)_\fp\bigr)\mid\fp\in\Spec R, \gcpd(M_\fp)\neq0\}.$$
Moreover, $\rgr(M)\leq\rgr(M,C)$ and the equality holds if
$\pd_R(\lambda M)<\infty$.
\end{thm}
\begin{proof}
We may assume that $\gkd_R(M)>0$. Set $n=\rgr(M,C)$. By Theorem
\ref{th1}, $\lambda M$ satisfies $\widetilde{S}_{n}$. Hence,
\begin{equation}\tag{\ref{th6}.1}
\depth_{R_\fp}(\lambda M)_\fp\geq\min\{\depth R_\fp,n\} \text{ for
all } \fp\in\Spec R
\end{equation}
Now let $\fp\in\Spec R$ with $\gcpd(M_\fp)\neq0$ so that $\depth
R_\fp>0$. As $M$ is a syzygy, we get $\depth_{R_\fp}( M_\fp)>0$.
Therefore, we have
$n\leq\rgr_{R_\fp}(M_\fp,C_\fp)\leq\gcpd(M_\fp)<\depth R_\fp$. It
follows from (\ref{th6}.1) that $\depth_{R_\fp}((\lambda M)_\fp)\geq
n$, and so $n\leq\inf\{\depth_{R_\fp}((\lambda M)_\fp)\mid
\fp\in\Spec R,
 \gd_{R_\fp}(M_\fp)\neq0\}.$

On the other hand, by the Lemma \ref{lem4}, if
$\fp\in\Ass_R(\Ext^n_R(M,C))$ then $\gcpd(M_\fp)\neq0$ and
 $\depth_{R_\fp}((\lambda M)_\fp)=\rgr(M,C)$ and so the assertion holds.

As $\lambda M=\Omega\Tr M$, $\Tr M\in\mathcal{A}_C$ by Example
\ref{example1}(i). Note that $M\approx\Tr\Tr M$. Hence it is enough
to replace $M$ by $\Tr M$ in the Theorem \ref{th5} for the last
part.
 \end{proof}
For a subset $X$ of $\Spec R$, we say that $M$ is of $\gc$-dimension
zero on $X$, if $\gcpd(M_\fp)=0$ for all $\fp$ in $X$.
\begin{prop}
Let $M$ be a horizontally linked $R$--module of finite and positive
$\gc$-dimension and $\lambda M\in\mathcal{A}_C$. Set
$t_M=\rgr(M,C)+\rgr(\lambda M)$, then $M$ is of $\gc$-dimension zero
on $X^{t_M-1}(R).$
\end{prop}
\begin{proof}
Set $n=\rgr(\lambda M)$. Assume contrarily that $\gcpd(M_\fp)\neq0$
for some $\fp\in X^{t_M-1}(R)$. Note that, by Theorem \ref{th1}(i),
we have
$$\depth_{R_\fp}(M_\fp)\geq \min\{\depth R_\fp,n\}$$
for all $\fp\in\Spec R$. If $\depth R_\fp\leq n$, then
$\gcpd(M_\fp)=\depth R_\fp-\depth_{R_\fp}(M_\fp)=0$, which is a
contradiction. If $\depth R_\fp>n$, then $\depth_{R_\fp}(M_\fp)\geq
n$. Therefore,
$$t_M-n\leq\rgr(M_\fp,C_\fp)\leq\gcpd(M_\fp)=\depth R_\fp-\depth_{R_\fp}(M_\fp)\leq t_M-n-1,$$
which is a contradiction.
\end{proof}
\section{reduced $\gc$-perfect modules}
Let $M$ be an $R$--module of finite positive $\gc$-dimension. The
following inequalities are well-known
$$\gr_R(M)\leq\rgr(M,C)\leq\gkd_R(M).$$
In the literature, $M$ is called $\gc$-\emph{perfect} if
$\gr_R(M)=\gkd_R(M)$. In this section we are interested in the case
where $\rgr(M,C)=\gkd_R(M)$ which we call $M$ reduced
$\gc$--perfect. So we bring the following definition.
\begin{dfn}\label{rgr}
\emph{Let $M$ be an $R$--module of finite $\gc$-dimension, we say
that $M$ is \emph{reduced $\gc$-perfect} if its $\gc$-dimension is
equal to its reduced grade with respect to $C$, i.e.
$\rgr(M,C)=\gkd_R(M)$.}
\end{dfn}
Note that every reduced G-perfect module has a finite and positive
$\gc$-dimension. It is obvious that $\Ext^{\tiny{\rgr_R(M,C)}}_R
(M,C)$ is the only non-zero module among all $\Ext^i_R(M,C)$ for
$i>0$.

Let $R$ be a Gorenstein local ring. Following \cite{H}, an
$R$--module $M$ is said to be an \emph{Eilenberg-Maclane} module, if
$\hh^i_\fm(M)=0$ for all $i\neq\depth_R(M), \dim_R(M)$. Thus reduced
$\gc$-perfect modules can be viewed as a generalization of
Eilenberg-MacLane modules. The following is a generalization of
\cite[Theorem 3.3]{DS}.
\begin{thm}\label{th4}
Let $R$ be a Cohen-Macaulay local ring of dimension $d$. If $M$ is
reduced $\gc$-perfect of $\gc$-dimension $n$ and $\lambda
M\in\mathcal{A}_C$, then
$$\depth_R(M)+\depth_R(\lambda M)=\depth
R+\depth_R(\Ext^n_R(M,C)).$$
\end{thm}
\begin{proof}
Let $\cdots\rightarrow P_n\rightarrow\cdots\rightarrow
P_0\rightarrow M\rightarrow0$ be the minimal projective resolution
of $M$. As $\Ext^i_R(M,C)=0$ for $0<i<n$, we obtain the following
exact sequences (as mentioned in (4.72), (4.7.3) and (4,7,4)):
\begin{equation}\tag{\ref{th4}.1}
0\rightarrow\trk M\rightarrow
(P_2)^{\triangledown}\rightarrow\cdots\rightarrow
(P_{n})^\triangledown\rightarrow\trc_{n}M\rightarrow0,
\end{equation}
\begin{equation}\tag{\ref{th4}.2}
0\rightarrow\Ext^n_R(M,C)\rightarrow\trc_{n}M\rightarrow
L\rightarrow0
\end{equation}
\begin{equation}\tag{\ref{th4}.3}
0\rightarrow L\rightarrow \overset{m}{\oplus}
C\rightarrow\trc_{n+1}M \rightarrow0.
\end{equation}
As $\gkd_R(\Omega^nM)=0$, $\gkd_R(\trc_{n+1}M)=0$ by Remark
\ref{rem1}, hence $\gkd_R(L)=0$. In other words, $\depth_R(L)=d$. It
follows from the exact sequence (\ref{th4}.2) that
\begin{equation}\tag{\ref{th4}.4}
\depth_R(\Ext^n_R(M,C))=\depth_R(\trc_{n}M).
\end{equation}
 By \cite[Corollary
4.17]{AB}, $\gr_R(\Ext^n_R(M,C))\geq n$. Note that
$$\dim_R(\Ext^n_R(M,C))=d-\gr_R(\Ext^n_R(M,C))$$ over the
Cohen-Macaulay local ring $R$. Therefore we have
\[\begin{array}{rl}
\depth_R(\trc_{n}M)&=\depth_R(\Ext^n_R(M,C))\\
&\leq\dim_R(\Ext^n_R(M,C))\\
&\leq d-n.
\end{array}\]
Now by the exact sequence (\ref{th4}.1), it is easy to see that
\begin{equation}\tag{\ref{th4}.5}
\depth(\trk M)=\depth_R(\trc_{n}M)+n-1.
\end{equation}
Note that $\trk M\cong\Tr M\otimes_RC$ by Remark \ref{remark3}(i).
As $\lambda M$ is the first syzygy of $\Tr M$, it follows from
Example \ref{example1}(i) that $\Tr M\in\mathcal{A}_C$. By Lemma
\ref{lem2}, $\depth_R(\Tr M)=\depth(\trk M)$ and so
\begin{equation}\tag{\ref{th4}.6}
\depth_R(\lambda M)=\depth(\trk M)+1.
\end{equation}
Now the assertion is clear by (\ref{th4}.4),(\ref{th4}.5),
(\ref{th4}.6) and Theorem \ref{G3}.
\end{proof}
\begin{thm}\label{th7}
Assume that $M$ is a horizontally linked $R$--module with finite and
positive $\gkd_R(M)=n$ and that $\lambda M\in\mathcal{A}_C$. Then
$M$ is reduced $\gc$-perfect if and only if $\lambda M$ satisfies
$\widetilde{S}_n$.
\end{thm}
\begin{proof}
As $M$ has a positive and finite $\gc$-dimension $n$,
$\rgr(M,C)\leq\gkd_R(M)=n$. On the other hand, by Theorem \ref{th1},
$\lambda M$ satisfies $\widetilde{S}_n$ if and only if
$\rgr(M,C)\geq n$. Therefore the assertion is obvious.
\end{proof}
In the following, we present
an equivalent condition for an Eilenberg-Maclane horizontally linked
$R$--module in terms of its linked module.

\begin{cor}\label{cor1}
Assume that $R$ is a Cohen-Macaulay local ring of dimension $d$ and
that $M$ a horizontally-linked $R$--module with $\depth_R(M)= n<d$.
If $\gd_R(\lambda M)<\infty$, then $M$ is Eilenberg-Maclane module
if and only if $\lambda M$ satisfies $\widetilde{S}_{d-n}$.
\end{cor}
\begin{proof}
We may assume that $R$ is complete with the canonical module
$\omega_R$. By Example \ref{example1}(ii), $\lambda
M\in\mathcal{A}_{\omega_R}$. Now the assertion is clear by Theorem
\ref{th7} and the Local Duality Theorem \ref{th8}.
\end{proof}
Recall that an $R$--module $M$ of dimension $d\geq1$ is called a
\emph{ generalized Cohen-Macaulay} module if
$\ell(\hh^i_\fm(M))<\infty$ for all $i$, $0\leq i\leq d-1$, where
$\ell$ denotes the length. For an $R$--module $M$ over a local ring
$R$, it is well-known that $$\depth_R(M)=\inf\{i|
\hh^i_\fm(M)\neq0\}.$$ Hence, if $M$ is an Eilenberg-Maclane
$R$--module which is not maximal Cohen-Macaulay, then
$\cc(M)=\depth_R(M)$. We end the paper by the following result which
is an immediate consequence of Theorem \ref{cor3}.
\begin{cor}\label{cor4}
Let $R$ be a Cohen-Macaulay local ring of dimension $d$ with
canonical module $\omega_R$ and let $M$ be a horizontally linked
$R$--module which is not Cohen-Macaulay $R$--module. Assume that $M$
is an Eilenberg-Maclane module and that $\gd_R(\lambda
M)<\infty$(e.g. $R$ is Gorenstein). Then $M$ is generalized
Cohen-Macaulay if and only if $\depth_{R_\fp}((\lambda
M)_\fp)+\depth_R(M)> d$ for all $\fp\in\cm(M)\setminus\{\fm\}$ and
$\depth_R(\lambda M)+\depth_R(M)=d$.
\end{cor}
{\bf Acknowledgement.} The authors owe their progress on this paper
to Jan Strooker who read the first draft, shared his ideas with us
and gave many encouraging comments. They also thank the referee for
giving invaluable corrections and comments.
\bibliographystyle{amsplain}

\end{document}